\documentclass[a4paper, 12pt]{amsart}

\usepackage{amsmath, amssymb, amscd, enumitem, nccmath, framed, stmaryrd}

\usepackage{etoolbox}

\makeatletter
\let\old@tocline\@tocline
\let\section@tocline\@tocline
\newcommand{\subsection@dotsep}{4.5}
\newcommand{\subsubsection@dotsep}{4.5}
\patchcmd{\@tocline}
  {\hfil}
  {\nobreak
     \leaders\hbox{$\m@th
        \mkern \subsection@dotsep mu\hbox{.}\mkern \subsection@dotsep mu$}\hfill
     \nobreak}{}{}
\let\subsection@tocline\@tocline
\let\@tocline\old@tocline

\patchcmd{\@tocline}
  {\hfil}
  {\nobreak
     \leaders\hbox{$\m@th
        \mkern \subsubsection@dotsep mu\hbox{.}\mkern \subsubsection@dotsep mu$}\hfill
     \nobreak}{}{}
\let\subsubsection@tocline\@tocline
\let\@tocline\old@tocline

\let\old@l@subsection\l@subsection
\let\old@l@subsubsection\l@subsubsection

\def\@tocwriteb#1#2#3{%
  \begingroup
    \@xp\def\csname #2@tocline\endcsname##1##2##3##4##5##6{%
      \ifnum##1>\c@tocdepth
      \else \sbox\z@{##5\let\indentlabel\@tochangmeasure##6}\fi}%
    \csname l@#2\endcsname{#1{\csname#2name\endcsname}{\@secnumber}{}}%
  \endgroup
  \addcontentsline{toc}{#2}%
    {\protect#1{\csname#2name\endcsname}{\@secnumber}{#3}}}%

\newlength{\@tocsectionindent}
\newlength{\@tocsubsectionindent}
\newlength{\@tocsubsubsectionindent}
\newlength{\@tocsectionnumwidth}
\newlength{\@tocsubsectionnumwidth}
\newlength{\@tocsubsubsectionnumwidth}
\newcommand{\settocsectionnumwidth}[1]{\setlength{\@tocsectionnumwidth}{#1}}
\newcommand{\settocsubsectionnumwidth}[1]{\setlength{\@tocsubsectionnumwidth}{#1}}
\newcommand{\settocsubsubsectionnumwidth}[1]{\setlength{\@tocsubsubsectionnumwidth}{#1}}
\newcommand{\settocsectionindent}[1]{\setlength{\@tocsectionindent}{#1}}
\newcommand{\settocsubsectionindent}[1]{\setlength{\@tocsubsectionindent}{#1}}
\newcommand{\settocsubsubsectionindent}[1]{\setlength{\@tocsubsubsectionindent}{#1}}

\renewcommand{\l@section}{\section@tocline{1}{\@tocsectionvskip}{\@tocsectionindent}{}{\@tocsectionformat}}%
\renewcommand{\l@subsection}{\subsection@tocline{2}{\@tocsubsectionvskip}{\@tocsubsectionindent}{}{\@tocsubsectionformat}}%
\renewcommand{\l@subsubsection}{\subsubsection@tocline{3}{\@tocsubsubsectionvskip}{\@tocsubsubsectionindent}{}{\@tocsubsubsectionformat}}%
\newcommand{\@tocsectionformat}{}
\newcommand{\@tocsubsectionformat}{}
\newcommand{\@tocsubsubsectionformat}{}
\expandafter\def\csname toc@1format\endcsname{\@tocsectionformat}
\expandafter\def\csname toc@2format\endcsname{\@tocsubsectionformat}
\expandafter\def\csname toc@3format\endcsname{\@tocsubsubsectionformat}
\newcommand{\settocsectionformat}[1]{\renewcommand{\@tocsectionformat}{#1}}
\newcommand{\settocsubsectionformat}[1]{\renewcommand{\@tocsubsectionformat}{#1}}
\newcommand{\settocsubsubsectionformat}[1]{\renewcommand{\@tocsubsubsectionformat}{#1}}
\newlength{\@tocsectionvskip}
\newcommand{\settocsectionvskip}[1]{\setlength{\@tocsectionvskip}{#1}}
\newlength{\@tocsubsectionvskip}
\newcommand{\settocsubsectionvskip}[1]{\setlength{\@tocsubsectionvskip}{#1}}
\newlength{\@tocsubsubsectionvskip}
\newcommand{\settocsubsubsectionvskip}[1]{\setlength{\@tocsubsubsectionvskip}{#1}}

\patchcmd{\tocsection}{\indentlabel}{\makebox[\@tocsectionnumwidth][l]}{}{}
\patchcmd{\tocsubsection}{\indentlabel}{\makebox[\@tocsubsectionnumwidth][l]}{}{}
\patchcmd{\tocsubsubsection}{\indentlabel}{\makebox[\@tocsubsubsectionnumwidth][l]}{}{}

\newcommand{\@sectypepnumformat}{}
\renewcommand{\contentsline}[1]{%
  \expandafter\let\expandafter\@sectypepnumformat\csname @toc#1pnumformat\endcsname%
  \csname l@#1\endcsname}
\newcommand{\@tocsectionpnumformat}{}
\newcommand{\@tocsubsectionpnumformat}{}
\newcommand{\@tocsubsubsectionpnumformat}{}
\newcommand{\setsectionpnumformat}[1]{\renewcommand{\@tocsectionpnumformat}{#1}}
\newcommand{\setsubsectionpnumformat}[1]{\renewcommand{\@tocsubsectionpnumformat}{#1}}
\newcommand{\setsubsubsectionpnumformat}[1]{\renewcommand{\@tocsubsubsectionpnumformat}{#1}}
\renewcommand{\@tocpagenum}[1]{%
  \hfill {\mdseries\@sectypepnumformat #1}}

\let\oldappendix\appendix
\renewcommand{\appendix}{%
  \leavevmode\oldappendix%
  \addtocontents{toc}{%
    \protect\settowidth{\protect\@tocsectionnumwidth}{\protect\@tocsectionformat\sectionname\space}%
    \protect\addtolength{\protect\@tocsectionnumwidth}{2em}}%
}
\makeatother

\makeatletter
\settocsectionnumwidth{2em}
\settocsubsectionnumwidth{2.5em}
\settocsubsubsectionnumwidth{3em}
\settocsectionindent{1pc}%
\settocsubsectionindent{\dimexpr\@tocsectionindent+\@tocsectionnumwidth}%
\settocsubsubsectionindent{\dimexpr\@tocsubsectionindent+\@tocsubsectionnumwidth}%
\makeatother

\settocsectionvskip{10pt}
\settocsubsectionvskip{0pt}
\settocsubsubsectionvskip{0pt}
    
\settocsectionformat{\bfseries}
\settocsubsectionformat{\mdseries}
\settocsubsubsectionformat{\mdseries}
\setsectionpnumformat{\bfseries}
\setsubsectionpnumformat{\mdseries}
\setsubsubsectionpnumformat{\mdseries}

\let\oldtableofcontents\tableofcontents
\renewcommand{\tableofcontents}{%
  \vspace*{-\linespacing}
  \oldtableofcontents}

\usepackage[tikz]{mdframed}
\usepackage[all]{xy}
\usepackage{multirow}

\usepackage{leftidx}

\newtheorem{thm}{Theorem}[section] 
\newtheorem{lem}[thm]{Lemma}

\newcommand{\G}{\mathbb{G}}

\newcommand{\ds}{\displaystyle}
\newcommand{\Exp}{{\rm Exp}}
\newcommand{\GL}{{\rm GL}}
\newcommand{\id}{{\rm id}}
\newcommand{\Mat}{{\rm Mat}}
\newcommand{\mf}{\mathfrak}

\newcommand{\wt}{\widetilde}

\newcommand{\indep}{{\rm indep}}

\newcommand{\sfP}{{\sf P}}

\newcommand{\circled}[1]{\raise0.1ex\hbox{\textcircled{\scriptsize{\raise0.2ex\hbox{#1}}}}}
\newcommand{\ouparrow}{\raise0.1ex\hbox{\textcircled{\scriptsize{\raise0.1ex\hbox{$\uparrow$}}}}}
\newcommand{\orightarrow}{ \raise0.1ex\hbox{\textcircled{\scriptsize{\raise0.23ex\hbox{\hspace{0.1mm}$\rightarrow$}}}} }

\allowdisplaybreaks[4]

\makeatletter

\renewcommand\section{\@startsection{section}{1}%
  \z@{-.5\linespacing\@plus-.7\linespacing}{.5\linespacing}%
  {\bf \Large {\normalfont\scshape}}}

\renewcommand\subsection{\@startsection{subsection}{2}%
  \z@{-.5\linespacing\@plus-.7\linespacing}{.5\linespacing}%
  {\bf \large {\normalfont\scshape}}}

\renewcommand\subsubsection{\@startsection{subsubsection}{3}%
  \z@{.5\linespacing\@plus.7\linespacing}{.5\linespacing}
   {\bf  {\normalfont\scshape}}}
\makeatother

\makeatletter
\@namedef{subjclassname}{{\rm 2020} Mathematics Subject Classification}
\makeatother

\begin{document}

\title{
Birational equivalence of exponential matrices}
\subjclass[2020]{Primary 15A54, Secondary 14L30}
\keywords{Matrix theory, Algebraic geometry}
\author[Ryuji Tanimoto]{Ryuji Tanimoto}
\address{Faculty of Education, Shizuoka University, 836 Ohya, Suruga-ku, Shizuoka 422-8529, Japan} 
\email{tanimoto.ryuji@shizuoka.ac.jp}
\maketitle

\begin{abstract}
In this article, we consider birational equivalence of exponential matrices. 
In characteristic zero, we give a birational classification of exponential matrices of 
size $n$-by-$n$ $(n \geq 2)$, which consists of two types. 
And in positive characteristic, we give birational classifications 
of exponential matrices of sizes two-by-two and three-by-three, respectively. 
\end{abstract}

%

\section*{Introduction}

In this article, we consider birational equivalence of exponential matrices. 
We can rephrase exponential matrices as $\G_a$-actions on projective spaces. 
In fact, there exists one-to-one correspondence between 
exponential matrices and $\G_a$-actions on projective spaces 
(see \cite[Theorem 0.2]{Tanimoto 2025}). 
In \cite{Gurjar-Masuda-Miyanishi}, 
unipotent group actions on projective varieties are studied, 
and in part $\G_a$-actions on the $n$-dimensional projective space $\mathbb{P}^n$ in characteristic zero are classified biregularly. 
Even in positive characteristic, it is natural to classify $\G_a$-actions on 
$\mathbb{P}^n$ biregularly. 
But it is too hard to give the biregular classification. 
So, it is natural to classify $\G_a$-actions on $\mathbb{P}^n$ birationally.  
This leads us to study exponential matrices birationally.

Birational equivalence of exponential matrices is defined, as follows: 
Let $k$ be an algebraically closed field of arbitrary characteristic $p$, 
let $\G_a$ denote the additive group of $k$, 
let $k[T]$ be the polynomial ring in one variable over $k$ 
and let $\Mat(n, k[T])$ denote the set of all $n \times n$ matrices 
whose entries belong to $k[T]$. 
Let $A_i(T)$ $(i = 1, 2)$ be exponential matrices of $\Mat(n, k[T])$ 
and let $\mu_i : \G_a \times \mathbb{P}^{n - 1} \to \mathbb{P}^{n - 1}$ 
$(i = 1, 2)$ be the $\G_a$-actions corresponding to $A_i(T)$ $(i = 1, 2)$, respectively. 
We say that the two exponential matrices $A_1(T)$ and $A_2(T)$ 
are {\it birationally equivalent} if there exists a birational transformation $\sigma$ 
of $\mathbb{P}^{n - 1}$ such that the following diagram is commutative: 
\[
\xymatrix@R=36pt@C=36pt@M=6pt{
 \G_a \times \mathbb{P}^{n - 1} \ar[r]^(.6){\mu_1} \ar@{-->}[d]_{\id_{\G_a} \times \sigma}  & \mathbb{P}^{n - 1} \ar@{-->}[d]^\sigma \\
 \G_a \times \mathbb{P}^{n - 1} \ar[r]^(.6){\mu_2}  & \mathbb{P}^{n - 1}
}
\]
If $A_1(T)$ and $A_2(T)$ are birationally equivalent, we write 
\[
 A_1(T) \quad \overset{\rm bir}{\sim} \quad A_2(T) . 
\]
Let $\Mat(n, k[T])^E$ denote the set of all exponential matrices of 
$\Mat(n, k[T])$. 
The relation $\overset{\rm bir}{\sim}$ on $\Mat(n, k[T])^E$ becomes an equivalence relation 
on $\Mat(n, k[T])^E$. 

In this article, we prove the following three thoerems:

\begin{thm}
Assume $p = 0$ and let $n \geq 2$. 
Then there exists a one-to-one correspondence between the set 
\[
\Mat(n, k[T])^E / \overset{\rm bir}{\sim}
\]
and the set 
\[
\left\{ 
\quad 
 I_n , \qquad 
\left(
\begin{array}{c c | c }
 1 & T &  \multirow{2}{*}{ $O_{2, \, n - 2}$ } \\
 0 & 1 & \\
\hline 
 \multicolumn{2}{c |}{ O_{n - 2, \, 2}}  & I_{n - 2} 
\end{array}
\right) 
\quad 
\right\} . 
\] 
\end{thm}

In positive characteristic, we denote by $\sfP$ the set of all $p$-polynomials. 
Let $(\, \sfP^n \,)^\indep$ denote the set consisting of all elements 
$(\alpha_1, \ldots, \alpha_n)$ of $\sfP^n$ such that 
$\alpha_1, \ldots, \alpha_n$ are linearly independent over $k$. 
Clearly, $\GL(n, k)$ acts on $ (\, \sfP^n \,)^\indep$. 
Two elements $\alpha$, $\beta$ of $(\, \sfP^n \,)^\indep$ are 
said to be {\it $\GL(n, k)$-equivalent} if there exists a regular matrix $P$ 
of $\GL(n, k)$ such that $\alpha = \beta P$. 
If $\alpha$ and $\beta$ of $(\, \sfP^n \,)^\indep$ are $\GL(n, k)$-equivalent, 
we write $\alpha \sim_{\GL(n, k)} \beta$. 
Clearly, the relation $\sim_{\GL(n, k)}$ on $(\, \sfP^n \,)^\indep$ 
is an equivalence relation on $(\, \sfP^n \,)^\indep$.  

Geometrically speaking, the following theorem gives a  birational classification 
of $\G_a$-actions on $\mathbb{P}^1$ in positive charactersitic:  

\begin{thm}
There exists a one-to-one correspondence between the set 
\[
\Mat(2, k[T])^E / \overset{ \rm bir }{\sim}
\]
and the set 
\[
\{\, 0 \,\} 
\; \; \sqcup \; \; 
\bigl( \; 
 (\, \sfP \,)^\indep / \sim_{\GL(1, k)} 
\; \bigr) . 
\] 
\end{thm}

Geometrically speaking, the following theorem gives a  birational classification 
of $\G_a$-actions on $\mathbb{P}^2$ in positive charactersitic:

\begin{thm}
There exists a one-to-one correspondence between the set 
\[
\Mat(3, k[T])^E / \overset{\rm bir}{\sim}
\]
and the set 
\[
 \{\, 0 \,\} 
\; \; \sqcup \; \; 
\bigl( \; 
 (\, \sfP \,)^\indep / \sim_{\GL(1, k)} 
\; \bigr)
\; \; \sqcup \; \; 
\bigl(\; 
(\, \sfP^2 \,)^\indep / \sim_{\GL(2, k)}
\; \bigr) . 
\]
\end{thm}

As seen from the above Theorems 0.2 and 0.3, 
these two classification results do not vary in accordance with 
the characteristic $p$, and show the same pattern. 
\\

We employ the following notation and definitions: 

For any commutative ring $R$, we denote by $\Mat(n, R)$ 
the set of all matrices of size $n$-by-$n$ whose entries belong to $R$. 
We denote by $I_n$ the identity matrix of $\Mat(n, R)$. 

If $p > 0$, a polynomial $f(T)$ of $k[T]$ is said to be a {\it $p$-polynomial} 
if $f(T)$ has the following form: 
\begin{align*}
 f(T) 
 = \sum_{i = 0} ^N a_i \, T^{p^i} , 
\qquad 
a_i \in k \quad \text{ for all \quad $0 \leq i \leq N$. } 
\end{align*}

A matrix $A(T)$ of $\Mat(n, k[T])$ 
is said to be an {\it exponential matrix} of $\Mat(n, k[T])$ 
if $A(T)$ satisfies the following conditions {\rm (1)} and {\rm (2)}: 
\begin{enumerate}[label = {\rm (\arabic*)}]
\item $A(T) \, A(T') = A(T + T')$ in $\Mat(n, k[T, T'])$, 
where $k[T, T']$ denotes the polynomial ring in two variables over $k$. 

\item $A(0) = I_n$. 
\end{enumerate}

\section{Preliminaries}

\begin{lem}
The following assertions {\rm (1)} and {\rm (2)} hold true: 
\begin{enumerate}[label = {\rm (\arabic*)}]
\item Let $A(T)$ and $B(T)$ be exponential matrices of $\Mat(n, k[T])^E$. 
If $A(T)$ is equivalent to $B(T)$, i.e., $P \, A(T) \, P^{-1} = B(T)$ for some $P \in \GL(n, k)$, 
then $A(T)$ is birationally equivalent to $B(T)$. 

\item Let $A(T)$ be an exponential matrix of $\Mat(n, k[T])^E$. 
Assume that $A(T)$ is birationally equivalent to $I_n$. 
Then $A(T) = I_n$. 
\end{enumerate} 
\end{lem}

\begin{proof}
The proofs of (1) and (2) are straightforward. 
%

\end{proof}

Assume that the characteristic of $k$ is zero and 
let $\mf{A}$ be a commutative $k$-algebra. 
A $k$-derivation $D$ of $\mf{A}$ is said to be {\it locally nilpotent} 
if for any element $a$ of $\mf{A}$ there exists an integer $n \geq 0$ such that 
$D^\ell(a) = 0$ for all $\ell \geq n$. 

Given a locally nilpotent derivation $D$ of $\mf{A}$, we can define a $k$-algebra homomorphism 
$\varphi_{D, \, T} : \mf{A} \to \mf{A}[T]$ as 
\[
 \varphi_{D, \, T}(f) 
 := 
\sum_{i \geq 0} \frac{D^i(f)}{i!} \, T^i , 
\]
where $\mf{A}[T]$ denotes the polynomial ring in one variable over $\mf{A}$.

Let $\mf{A}[T, T']$ denote the polynomial ring in two variables over $k$. 
We have 
\[
 \varphi_{D, \, T + T'} = \varphi_{D, \, T} \circ \varphi_{D, \, T'} . 
\]
The $k$-algebra homomorphism $\varphi_{D, \, T + T'} : \mf{A} \to \mf{A}[T, T']$ 
in the left hand side is a natural extension of $\varphi_{D, \, T + T' } : \mf{A} \to \mf{A}[T + T']$, 
and the $k$-algebra homomorphism $\varphi_{D, \, T} : \mf{A}[T'] \to \mf{A}[T'][T]$ 
in the right hand side is a natural exstension of $\varphi_{D, \, T} : \mf{A} \to \mf{A}[T]$, 
where we let $\varphi_{D, \, T}(T') := T'$. 
The following diagram is commutative:  
\[
\xymatrix@R=36pt@C=36pt@M=6pt{
 & \mf{A} \ar[r]^{\varphi_{D, \, T}}   \ar@{^(->}[d]
 & \mf{A}[T]   \ar@{^(->}[d]  \\
 \mf{A} \ar[r]^{\varphi_{D, \, T'} } \ar@{-->}@/_24pt/[rr]_{\varphi_{D, \, T + T'}} 
 \ar@/_24pt/[rrd]_{\varphi_{D, \, T + T'}}
 & \mf{A}[T']  \ar@{-->}[r]^{\varphi_{D, \, T}}
 & \mf{A}[T, T'] \\
 & 
 & \mf{A}[T + T'] \ar@{^(->}[u]
}
\]

For any $t \in k$, we can define a map $\epsilon_t : \mf{A}[T] \to \mf{A}$ as 
\[
 \epsilon_t(\, f(T) \, ) := f(t) . 
\]
We can also define a map $\varphi_{D, \, t} : \mf{A} \to \mf{A}$ as 
\[
  \varphi_{D, \, t} := \epsilon_t \circ \varphi_{D, \, T} . 
\]
Clearly, $\varphi_{D, \, t}$ is a $k$-algebra automorphism of $\mf{A}$.

\section{Proof of Theorem 0.1}

In this section, we assume that the characteristic of $k$ is zero. 
Let $n \geq 2$ and let $k[x_1, \ldots, x_n]$ be the polynomial ring in $n$ variables over $k$. 
Let $D$ be a $k$-derivation of $k[x_1, \ldots, x_n]$ 
with the following form: 
\[
 D = 
x_1 \, \frac{\partial}{\partial x_2}
 + 
\sum_{i = 3}^n 
f_i \, 
\frac{\partial}{\partial x_i} , 
\]
where $f_i \in k[x_1, \ldots , x_{i - 1}] \, (\subset k[x_1, \ldots , x_n])$ for all $3 \leq i \leq n$ and 
each $f_i$ is zero or a homogenous polynomial of degree $1$. 
Since $D(x_1) = 0$, we can naturally extend the locally nilpotent derivation $D$ 
of $k[x_1, \ldots, x_n]$ to the locally nilpotent derivation $\wt{D}$ of $k[x_1, \ldots, x_n][1/x_1]$ 
defined by 
\[
 \wt{D}\left( \frac{a}{x_1^\ell} \right) 
 = 
\frac{D(a)}{x_1^\ell} , 
\]
where $a \in k[x_1 , \ldots , x_n]$ and $\ell \geq 0$. 
For simplicity, we write $D$ in place of $\wt{D}$. 
So, we can define a $\G_a$-action $\mu_D : \G_a \times \mathbb{P}^{n - 1} \to \mathbb{P}^{n - 1}$ as 
\[
 \mu_D(t , \; (x_1 : \cdots : x_n)) 
 := 
\left( \, 
 \varphi_{D, \, t}\left( \frac{x_1}{x_1}  \right) 
 \, : \, 
 \varphi_{D, \, t}\left( \frac{x_2}{x_1}  \right) 
 \, : \, 
 \cdots 
  \, : \, 
 \varphi_{D, \, t}\left( \frac{x_n}{x_1}  \right) 
\, \right) . 
\]

The following lemma states that the $\G_a$-action $\mu_D$ is birationally equivalent 
to the $\G_a$-action $\nu : \G_a \times \mathbb{P}^{n - 1} \to \mathbb{P}^{n - 1}$ 
defined by 
\[
  \nu(t , \; (y_1 : \cdots : y_n)) 
:= 
\left( \, 
 y_1
 \, : \, 
 y_2 + t \, y_1 
  \, : \, 
 y_3 
 \, : \, 
 \cdots 
  \, : \, 
 y_n 
\, \right) , 
\]
which corresponds to an exponential matrix 
\[
 \left(
\begin{array}{c c | c }
 1 & 0 &  \multirow{2}{*}{ $O_{2, \, n - 2}$ } \\
 T & 1 & \\
\hline 
 \multicolumn{2}{c |}{ O_{n - 2, \, 2}}  & I_{n - 2} 
\end{array}
\right) . 
\]

\begin{lem}
Let $s := x_2 / x_1$ and let $\sigma : \mathbb{P}^{n - 1} \dasharrow \mathbb{P}^{n - 1}$ be the 
rational map defined by 
\[
 \sigma(x_1 : \cdots : x_n)
 := 
\left( \, 
 1
 \, : \, 
 \frac{x_2}{x_1}  
 \, : \, 
\varphi_{D, \, - s}\left( \frac{x_3}{x_1}  \right) 
 \, : \, 
 \cdots 
  \, : \, 
 \varphi_{D, \, - s}\left( \frac{x_n}{x_1}  \right) 
\, \right) . 
\]
Then the following assertions {\rm (1)} and {\rm (2)} hold true: 
\begin{enumerate}[label = {\rm (\arabic*)}]
\item $\sigma$ is birational. 

\item The following diagram is commutative: 
\[
\xymatrix@R=36pt@C=36pt@M=6pt{
 \G_a \times \mathbb{P}^{n - 1} \ar[r]^(.55){\mu_D} \ar@{-->}[d]_{1 \times \sigma}
 & \mathbb{P}^{n - 1} \ar@{-->}[d]^\sigma \\
\G_a \times \mathbb{P}^{n - 1} \ar[r]^(.55)\nu 
 & \mathbb{P}^{n - 1}
}
\]
\end{enumerate} 
\end{lem}

\begin{proof}
(1) Let $A := k[\frac{x_2}{x_1}, \frac{x_3}{x_1} , \ldots, \frac{x_n}{x_1}]$. 
Clearly, $\wt{D}(A) \subset A$. 
So, the locally nilpotent derivation $\wt{D}$ of $k[x_1, \ldots , x_n][1/x_1]$ 
induces a locally nilpotent derivation $\delta$ of $A$. 
Clearly, $\delta(s) = 1$. 
So, we have $A = A^\delta[s]$ and $A^\delta = \varphi_{\delta, \, -s}(A)
 = \varphi_{D, \, -s}(A)$ 
(see \cite{van den Essen}). Thus 
\begin{align*}
A  
 & =  k\left[\,  
\varphi_{D, \, -s}\left( \frac{x_3}{x_1} \right) , \;  
 \ldots, 
\varphi_{D, \, -s}\left( \frac{x_n}{x_1} \right) , \; 
\frac{x_2}{x_1} 
\, \right] . 
\end{align*}

(2) Choose any element $(t, \; (x_1 : \cdots : x_n))$ of $\G_a \times \mathbb{P}^{ n - 1}$. 
Write
\[
 \sigma(x_1 : \cdots : x_n) = (y_1 : \cdots : y_n) . 
\]
We have 
\begin{align*}
 & (\sigma \circ \mu_D)(t, \; (x_1 : \cdots : x_n))  \\
 & \quad = \sigma 
\left( \, 
 1 
 \, : \, 
 \varphi_{D, \, t}\left( \frac{x_2}{x_1}  \right) 
 \, : \, 
 \cdots 
  \, : \, 
 \varphi_{D, \, t}\left( \frac{x_n}{x_1}  \right) 
\, \right) \\
 & \quad = 
\left( \, 
 1 
 \, : \, 
 \varphi_{D, \, t}\left( \frac{x_2}{x_1}  \right) 
 \, : \, 
\varphi_{D, \, -s}\left(   \varphi_{D, \, t}\left( \frac{x_3}{x_1}  \right)  \right)
 \, : \, 
 \cdots 
  \, : \, 
 \varphi_{D, \, -s}\left( \varphi_{D, \, t}\left( \frac{x_n}{x_1}  \right) \right) 
\, \right) \\
 & \quad = 
\left( \, 
 1 
 \, : \, 
 \frac{x_2 + t \, x_1}{x_1} 
 \, : \, 
\varphi_{D, \, -s}\left( \frac{x_3}{x_1}  \right) 
 \, : \, 
 \cdots 
  \, : \, 
 \varphi_{D, \, -s}\left(  \frac{x_n}{x_1}  \right) 
\, \right) \\
 & \quad = 
\left( \, 
 1 
 \, : \, 
 \frac{y_2}{y_1} + t  
 \, : \, 
 \frac{y_3}{y_1}   
 \, : \, 
 \cdots 
  \, : \, 
 \frac{y_n}{y_1}  
\, \right) \\
 & \quad = 
\left( \, 
 y_1 
 \, : \, 
 y_2 + t \, y_1  
 \, : \, 
 y_3 
 \, : \, 
 \cdots 
  \, : \, 
 y_n 
\, \right) \\
 & \quad = 
\bigl( \, \nu \circ (1 \times \sigma)  \, \bigr) (t, \; (x_1 : \cdots : x_n)) . 
\end{align*}

\end{proof}

We shall prove Theorem 0.1. 
Let $A(T)$ be an exponential matrix of $\Mat(n, k[T])$ and 
assume $A(T) \ne I_n$. 
There exists a unique nilpotent matrix $N$ of $\Mat(n, k)$ so that 
\[
 A(T) = \Exp_N(T), 
\]
where $\Exp_N(T)$ is defined as 
\[
 \Exp_N(T) := \sum_{i \geq 0} \frac{T^i}{i!} N^i . 
\]
Since $A(T) \ne I_n$, we have $N \ne O_n$. 
So, there exists a regular matrix $P$ of $\GL(n, k)$ so that 
\[
P \, N \, P^{-1} 
 = 
\left( 
\begin{array}{c c c c c | c}
 J(d_1) & O & \cdots & \cdots &  O & \multirow{5}{*}{$O$}\\
 O & J(d_2) & \ddots  &  & O & \\
 \vdots & \ddots & \ddots & \ddots & \vdots & \\
 \vdots & &  \ddots & J(d_{r - 1}) & O &\\
 O & \cdots & \cdots & O & J(d_r) & \\
\hline 
 \multicolumn{5}{c |}{O} & O  
\end{array}
\right) , 
\]
where each matrix $J(d_s) = (\, \nu_{i, j}^{(s)} \,)$ $(1 \leq s \leq r)$ satisfies  
\[
 \nu_{i, j}^{(s)} 
= 
\left\{
\begin{array}{r @{\qquad} l}
 1 & \text{ if \quad $i - j = 1$} , \\
 0 & \text{ otherwise} .  
\end{array}
\right.
\]
Let $M := P \, N \, P^{-1}$ and 
let $D$ be the $k$-linear derivation of $k[x_1, \ldots, x_n]$ defined by 
\[
\bigl( \, D(x_1), \, D(x_2) , \, \ldots, D(x_n) \, \bigr)
= 
(x_1, \, x_2 , \, \ldots, x_n) \; {^t}M . 
\]
We can express $D$ as 
\[
 D = 
x_1 \, \frac{\partial}{\partial x_2}
 + 
\sum_{i = 3}^n 
f_i \, 
\frac{\partial}{\partial x_i} , 
\]
where $f_i \in k[x_1, \ldots , x_{i - 1}] \, (\subset k[x_1, \ldots , x_n])$ for all $3 \leq i \leq n$ and 
each $f_i$ is zero or a homogenous polynomial of degree $1$. 
Let $B(T) := P \, A(T) \, P^{-1} = \Exp_{M}(T)$ and 
let $\tau : \G_a \times \mathbb{P}^{n -1} \to \mathbb{P}^{n - 1}$  
be the $\G_a$-action corresponding to $B(T)$. 
For all $(t, \, (x_1 \, : \, \cdots \, : \, x_n)) \in \G_a \times \mathbb{P}^{n - 1}$, 
we have 
\begin{fleqn}[8em] 
\begin{align*}
 & \tau( t, \, (x_1 \, : \, \cdots \, : \, x_n) ) \\
 & \qquad =  (x_1 \, : \, \cdots \, : \, x_n) \; {^t} B(t)  \\
 & \qquad =  (x_1 \, : \, \cdots \, : \, x_n) \; {^t} \, \Exp_{M}(t) \\
 & \qquad =  (x_1 \, : \, \cdots \, : \, x_n) \; \sum_{i \geq 0} \frac{t^i}{i!} \, ({^t}M)^i . 
\end{align*} 
\end{fleqn}
Since
\begin{fleqn}[8em]
\begin{align*}
&  (x_1, \, \ldots ,  \, x_n) \; \sum_{i \geq 0} \frac{t^i}{i!} \, ({^t}M)^i  \\
 & \qquad = 
\sum_{i \geq 0} \frac{t^i}{i!} \,  (x_1, \, \ldots ,  \, x_n) \, ({^t}M)^i  \\
 & \qquad = 
\sum_{i \geq 0} \frac{t^i}{i!} \, \bigl( \,  D^i(x_1), \, D^i(x_2), \, \ldots, \, D^i(x_n)  \, \bigr) \\
 & \qquad = 
\bigl( \, 
\varphi_{D, \, t}(x_1) , \, \ldots, \, \varphi_{D, \, t}(x_n) 
\, \bigr) , 
\end{align*}
\end{fleqn}
and since $\varphi_{D, \, t}(x_1) = x_1$, we have 
\begin{fleqn}[8em]
\begin{align*}
 & (x_1 \, : \, \cdots \, : \, x_n) \; \sum_{i \geq 0} \frac{t^i}{i!} \, ({^t}M)^i  \\
 & \qquad = \bigl( \, \varphi_{D, \, t}(x_1) \, : \, \cdots \, : \, \varphi_{D, \, t}(x_n) \, \bigr) \\ 
 &\qquad = \bigl( \, 
 1 \, : \, 
 \varphi_{D, \, t}(x_2)/ \varphi_{D, \, t}(x_1) 
 \, : \, 
 \cdots 
 \, : \, 
 \varphi_{D, \, t}(x_n)/ \varphi_{D, \, t}(x_1) 
\, \bigr) \\
 & \qquad = 
 \mu_D(t , \; (x_1 : \cdots : x_n)) . 
\end{align*} 
\end{fleqn}
Thus 
\[
 \tau =  \mu_D . 
\]
We know from Lemma 2.1 that 
\[
 B(T) \; \overset{\rm bir}{\sim}  
 \left(
\begin{array}{c c | c }
 1 & 0 &  \multirow{2}{*}{ $O_{2, \, n - 2}$ } \\
 T & 1 & \\
\hline 
 \multicolumn{2}{c |}{ O_{n - 2, \, 2}}  & I_{n - 2} 
\end{array}
\right) .  
\]
Since $A(T)$ is equivalent to $B(T)$, we can obtain 
\[
A(T) 
\; \, \overset{\rm bir}{\sim} \; \, 
\left(
\begin{array}{c c | c }
 1 & T &  \multirow{2}{*}{ $O_{2, \, n - 2}$ } \\
 0 & 1 & \\
\hline 
 \multicolumn{2}{c |}{ O_{n - 2, \, 2}}  & I_{n - 2} 
\end{array}
\right) . 
\]

\section{Proof of Theorem 0.2}

\begin{lem}
Let $A(T)$ and $B(T)$ be exponential matrices of $\Mat(2, k[T])^E$ with the form 
\[
 A(T) 
= 
\left(
\begin{array}{c c}
 1 & \alpha(T) \\
 0 & 1
\end{array}
\right) , 
\qquad 
 B(T) 
= 
\left(
\begin{array}{c c}
 1 & \beta(T) \\
 0 & 1
\end{array}
\right) 
\qquad 
(\, \alpha(T), \beta(T) \in \sfP \backslash \{ \, 0 \, \}  \,) .
\]
Then the following conditions {\rm (1)} and {\rm (2)} are equivalent: 
\begin{enumerate}[label = {\rm (\arabic*)}]
\item The exponential matrices $A(T)$ and $B(T)$ are birationally equivalent. 

\item There exists an element $\lambda$ of $k \backslash \{ 0  \}$ such that $\alpha(T) = \lambda \cdot \beta(T)$. 
\end{enumerate} 
\end{lem}

\begin{proof} 
We first prove the implication (1) $\Longrightarrow$ (2). 
There exists a birational map $\sigma : \mathbb{P}^1_k \dasharrow \mathbb{P}^1_k$ 
such that the following diagram is commutative: 
\[
\xymatrix@R=36pt@C=36pt@M=6pt{
 \G_a \times \mathbb{P}^1 \ar[r]^(.55){A(T)} \ar@{-->}[d]_{1 \times \sigma}
 & \mathbb{P}^1 \ar@{-->}[d]^\sigma \\
\G_a \times \mathbb{P}^1 \ar[r]^(.55){B(T)}
 & \mathbb{P}^1
}
\]
Since $\sigma$ is biregular, there exists a regular matrix $P$ of $\GL(2, k)$ such that 
\[
 \sigma(x_0 \, : \, x_1 ) = (x_0 \, : \, x_1 ) \, {^t}P . 
\]
We have $P \, A(T) = B(T) \, P$. 
This equality implies that condition (2) holds true.

We next prove (2) $\Longrightarrow$ (1). 
Let $\sigma : \mathbb{P}^1_k \dasharrow \mathbb{P}^1_k$ be the birational map 
defined by 
\[
\sigma (x_0 \, : \, x_1) = (x_0 \, : \, \lambda \, x_1) . 
\]
This birational map $\sigma$ gives an birational equivalence between $A(T)$ and $B(T)$. 
\end{proof}

We shall prove Theorem 0.2. 
Let $A(T)$ be an exponential matrix of $\Mat(2, k[T])^E$. 
By \cite[Lemma 1.8]{Tanimoto 2019}, 
there exists a regular matrix $P$ of $\GL(2, k)$ such that $P \, A(T) \, P^{-1}$ has the form 
\[
 P \, A(T) \, P^{-1} 
= 
\left(
\begin{array}{c c}
 1 & \alpha(T) \\
 0 & 1
\end{array}
\right) 
\qquad 
(\, \alpha(T) \in \sfP \,) . 
\]
By Lemma 3.1, we can obtain the desired one-to-one correspondence.

\section{Proof of Theorem 0.3}

We introduce three subsets $\mf{A}_{1, \, 2}^E$, $\mf{A}_{2, \, 1}^E$, $\mf{J}_{[3]}^E$ of 
$\Mat(3, k[T])^E$, as follows:

We denote by $\mf{A}_{1, \, 2}^E$ the set of all exponential matrices $A(T)$ with the form 
\begin{fleqn}[5em] 
\begin{align*}
 A(T) 
= 
\left(
\begin{array}{c c c}
 1 & \alpha_1(T) & \alpha_2(T) \\
 0 & 1 & 0 \\
 0 & 0 & 1
\end{array} 
\right) 
\qquad 
\bigl( \, 
\alpha_1(T) , \alpha_2(T) \in \sfP
\, \bigr) . 
\end{align*}
\end{fleqn}

We denote by $\mf{A}_{2, \, 1}^E$ the set of all exponential matrices $A(T)$ with the form 
\begin{fleqn}[5em] 
\begin{align*}
 A(T) 
= 
\left(
\begin{array}{c c c}
 1 & 0 & \alpha_2(T) \\
 0 & 1 & \alpha_1(T) \\
 0 & 0 & 1
\end{array} 
\right) 
\qquad 
\bigl( \, 
\alpha_1(T) , \alpha_2(T) \in \sfP
\, \bigr) . 
\end{align*}
\end{fleqn}

We denote by $\mf{J}_{[3]}^E$ the set of all exponential matrices $A(T)$ with the form 
\begin{fleqn}[5em] 
\begin{align*}
 A(T) 
= 
\left(
\begin{array}{c c c}
 1 & \alpha_1(T) & \frac{1}{2} \, \alpha_1(T)^2 + \alpha_2(T) \\
 0 & 1 & \alpha_1(T) \\
 0 & 0 & 1
\end{array} 
\right) 
\qquad 
\bigl( \, 
\alpha_1(T) , \alpha_2(T) \in \sfP , \; \alpha_1(T) \ne 0 
\, \bigr) . 
\end{align*}
\end{fleqn} 
The set $\mf{J}_{[3]}^E$ is defined only when $p \geq 3$.

\begin{lem}
The following assertions {\rm (1)} and {\rm (2)} hold true: 
\begin{enumerate}[label = {\rm (\arabic*)}]
\item If $p = 2$, we have 
$\Mat(3, k[T])^E 
  \; \leadsto \; 
 \mf{A}_{1, \, 2}^E \; \cup \; \mf{A}_{2, \, 1}^E
 $. 

\item If $p \geq 3$, we have 
$
 \Mat(3, k[T])^E 
  \; \leadsto \; 
 \mf{A}_{1, \, 2}^E \; \cup \; \mf{A}_{2, \, 1}^E \; \cup \; \mf{J}_{[3]}^E
 $. 
\end{enumerate} 
\end{lem}

\begin{proof} 

See \cite[Theorem 6.1]{Tanimoto 2008}. 
\end{proof}

Let $S$ and $S'$ be non-empty subsets $\Mat(n, k[T])^E$. 
We write
\[
  S \quad \overset{\rm bir}{\leadsto} \quad S'
\]
if for any exponential matrix $A(T)$ of $S$ 
there exists an exponential matrix $B(T)$ of $S'$ such that 
$A(T)$ is birationally equivalent to $B(T)$. 

We write 
\[
  S \quad \leadsto \quad S'
\]
if for any exponential matrix $A(T)$ of $S$ 
there exists an exponential matrix $B(T)$ of $S'$ such that 
$A(T)$ is equivalent to $B(T)$. 
So, if $S \leadsto S'$, then $S \overset{\rm bir}{\leadsto} S'$.

\begin{lem}
For all subsets $S$, $S'$, $S''$, $T$, $T'$ of $\Mat(n, k[T])^E$, 
the following assertions {\rm (1)}, {\rm (2)}, {\rm (3)} hold true: 
\begin{enumerate}[label = {\rm (\arabic*)}]
\item $S \overset{\rm bir}{\leadsto} S$. 

\item If $S \overset{\rm bir}{\leadsto} S'$ and $S' \overset{\rm bir}{\leadsto} S''$, 
then $S \overset{\rm bir}{\leadsto} S''$. 

\item If $S \overset{\rm bir}{\leadsto} S'$ and $T \overset{\rm bir}{\leadsto} T'$, 
then $S \cup T \overset{\rm bir}{\leadsto} S' \cup T'$. 
\end{enumerate} 
\end{lem}

\begin{proof}
The proof is straightforward. 
\end{proof}

\subsection{On $\mf{A}_{1, \, 2}^E$, $\mf{A}_{2, \, 1}^E$, $\mf{J}_{[3]}^E$}

We denote by $(\, \mf{A}_{1, \, 2}^E \,)^{\indep} $ the set of all exponential matrices $A(T)$ 
of $\mf{A}_{1, \, 2}^E$ with the form 
\[
 A(T) 
= 
\left(
\begin{array}{c c c}
 1 & \alpha_1(T) & \alpha_2(T) \\
 0 & 1 & 0 \\
 0 & 0 & 1
\end{array} 
\right) 
\qquad 
\left(
\begin{array}{l}
\alpha_1(T) , \alpha_2(T) \in \sfP , \\
\text{$\alpha_1(T)$ and $\alpha_2(T)$ are linearly independent over $k$}
\end{array}
\right) . 
\]

We denote by $\mf{A}_{1, \, 1}^E$ the set of all exponential matrices $A(T)$ with the form 
\[
 A(T) 
= 
\left(
\begin{array}{c c c}
 1 & 0 & \alpha(T) \\
 0 & 1 & 0 \\
 0 & 0 & 1
\end{array} 
\right) 
\qquad 
\bigl( \, 
\alpha(T)  \in \sfP
\, \bigr) . 
\]
We let $(\, \mf{A}_{1, \, 1}^E \,)^\indep := \mf{A}_{1, \, 1}^E \backslash \{\, I_3 \,\}$.

\begin{lem}
The following assertions {\rm (1)}, {\rm (2)}, {\rm (3)} hold true: 
\begin{enumerate}[label = {\rm (\arabic*)}]
\item 
$
\mf{A}_{1, \, 2}^E 
\; \, \leadsto \; \, 
(\, \mf{A}_{1, \, 2}^E \,)^{\indep} \, \sqcup \, (\, \mf{A}_{1, \, 1}^E \,)^\indep \, \sqcup \, \{\, I_3 \,\}
$.

\item $\mf{A}_{2, \, 1}^E  
\; \, \overset{\rm bir}{\leadsto} \; \, 
\mf{A}_{1, \, 1}^E$.

\item $\mf{J}_{[3]}^E \; \, \overset{\rm bir}{\leadsto} \; \, \mf{A}_{2, \, 1}^E$. 
\end{enumerate} 
\end{lem}

\subsubsection{Proof of (1)} 

Let $A(T) \in \mf{A}_{1, \, 2}^E$ and write 
\[
 A(T) 
= 
\left(
\begin{array}{c c c}
 1 & \alpha_1(T) & \alpha_2(T) \\
 0 & 1 & 0 \\
 0 & 0 & 1
\end{array} 
\right) 
\qquad 
\bigl( \, 
\alpha_1(T) , \alpha_2(T) \in \sfP
\, \bigr) . 
\]
Let $W$ be the subspace of $\sfP$ spanned by $\alpha_1(T)$ and $\alpha_2(T)$ over $k$. 
Let $d$ be the dimension of $W$ as a $k$-vector space. 
If $d = 2$, then $A(T) \in (\, \mf{A}_{1, \, 2}^E \,)^\indep$. 
If $d = 1$, then threre exists a regular matrix $P$ of $\GL(2, k)$ such that 
$P A(T) P^{-1} \in (\, \mf{A}_{1, \, 1}^E \,)^\indep$. 
If $d = 0$, then $A(T) = I_3$.

\subsubsection{Proof of (2) }

Let $A(T)$ be a non-zero matrix $A(T) \in \mf{A}_{2, \, 1}^E$ 
and write 
\[
 A(T) 
= 
\left(
\begin{array}{c c c}
 1 & 0 & \alpha_2(T) \\
 0 & 1 & \alpha_1(T) \\
 0 & 0 & 1
\end{array} 
\right) 
\qquad 
\bigl( \, 
\alpha_1(T) , \alpha_2(T) \in \sfP
\, \bigr) . 
\]
We can define the {\it degree} $\deg A(T)$ of 
$A(T)$ as 
\[
 \deg \, A(T) := \max\{\, \deg \alpha_1(T), \; \deg \alpha_2(T) \,\} , 
\]
where $\deg 0 := - \infty$.

\begin{lem}
Let $A(T)$ be an exponential matrix of $\Mat(3, k[T])^E$ with the form 
\[
 A(T) 
= 
\left(
\begin{array}{c c c}
 1 & 0 & \alpha_2(T) \\
 0 & 1 & \alpha_1(T) \\
 0 & 0 & 1
\end{array} 
\right) 
\qquad 
\bigl( \, 
\alpha_1(T) , \alpha_2(T) \in \sfP
\, \bigr) . 
\]
Assume $\alpha_1(T) \ne 0$ and $\alpha_2(T) \ne 0$. 
Then there exists an exponential matrix $B(T)$ of $\mf{A}_{2, \,1}^E$ 
such that $A(T)$ is birationally equivalent to $B(T)$ 
and its form 
\[
 B(T) 
= 
\left(
\begin{array}{c c c}
 1 & 0 & \beta_2(T) \\
 0 & 1 & \beta_1(T) \\
 0 & 0 & 1
\end{array} 
\right) 
\qquad 
\bigl( \, 
\beta_1(T) , \beta_2(T) \in \sfP
\, \bigr) 
\]
satisfies one of the following conditions {\rm (1)}, {\rm (2)}, {\rm (3)}: 
\begin{enumerate}[label = {\rm (\arabic*)}]
\item $\beta_1 \ne 0$, $\beta_2 \ne 0$ and $\deg A(T) > \deg B(T)$. 

\item $\beta_1 \ne 0$, $\beta_2 = 0$. 

\item $\beta_1 = 0$, $\beta_2 \ne 0$. 
\end{enumerate} 
\end{lem}

\begin{proof}
Write 
\begin{align*}
\begin{array}{r @{\,} l @{\qquad} l }
 \alpha_1(T) 
& = \ds \sum_{i = 0}^e c_i \, T^{p^i}  & (\, c_0, \ldots, c_e \in k , \quad c_e \ne 0 \,) , \\
 \alpha_2(T) 
& = \ds \sum_{i = 0}^f d_i \, T^{p^i}  & (\, d_0, \ldots, d_f \in k , \quad d_f \ne 0 \,) . 
\end{array}
\end{align*}
One of the following cases (i), (ii), (iii) can occur: 
\begin{enumerate}[label = {\rm (\roman*)}]
\item $\deg \alpha_1(T) < \deg \alpha_2(T)$. 

\item $\deg \alpha_1(T) = \deg \alpha_2(T)$. 

\item $\deg \alpha_1(T) > \deg \alpha_2(T)$. 
\end{enumerate} 
\medskip 

\noindent 
{\bf Case (i).}  We have $f > e$. 
Let $\lambda(T) := ( d_f / c_e^{p^{f - e}} ) \, T^{p^{f - e}} $. 
Let $\beta_1(T) := \alpha_1(T)$ and 
$\beta_2(T) := \alpha_2(T) - \lambda(T) \circ \alpha_1(T)$. 
Clearly, $\lambda(T), \beta_1(T), \beta_2(T) \in \sfP$. 
We can show $\deg \beta_2(T) < \deg \alpha_2(T)$. 
Let $\sigma : \mathbb{P}^2_k \dasharrow \mathbb{P}^2_k$ be the birational 
map defined by 
\[
 \sigma(x_0 \, : \, x_1 \, : \, x_2) 
 := 
\left( \, 
\frac{x_0}{x_2} - \lambda\left( \frac{x_1}{x_2} \right)
 \, : \,  
\frac{x_1}{x_2} 
 \, : \,  
 1 
\, \right) . 
\]
The following diagram is commutative; 
\[
\xymatrix@R=36pt@C=36pt@M=6pt{
 \G_a \times \mathbb{P}^2_k \ar[r]^(.6){A(T)} \ar@{-->}[d]_{1 \times \sigma} 
 & \mathbb{P}^2_k \ar@{-->}[d]^\sigma \\
 \G_a \times \mathbb{P}^2_k \ar[r]_(.6){B(T)}
 & \mathbb{P}^2_k 
}
\]
If $\beta_2(T) \ne 0$, then $B(T)$ satisfies condition (1). 
If $\beta_2(T) = 0$, then $B(T)$ satisfies condition (2). 
\medskip

\noindent
{\bf Cases (ii) and (iii).} The proofs are similar to the proof of Case (i). 
So, we omit the proofs.

\end{proof}

We shall give the proof of assertion (2) of Lemma 4.3. 
Let $A(T)$ be an exponential matrix of $\Mat(3, k[T])^E$ with the form 
\[
 A(T) 
= 
\left(
\begin{array}{c c c}
 1 & 0 & \alpha_2(T) \\
 0 & 1 & \alpha_1(T) \\
 0 & 0 & 1
\end{array} 
\right) 
\qquad 
\bigl( \, 
\alpha_1(T) , \alpha_2(T) \in \sfP
\, \bigr) . 
\]
Assume $\alpha_1(T) \ne 0$ and $\alpha_2(T) \ne 0$. 
Then $A(T)$ is birationally equivalent to an exponential matrix $B(T)$ with the form 
\[
 B(T) 
= 
\left(
\begin{array}{c c c}
 1 & 0 & \beta(T) \\
 0 & 1 & 0 \\
 0 & 0 & 1
\end{array} 
\right) 
\qquad 
\bigl( \, 
\beta(T) \in \sfP \backslash \{ \, 0 \, \} 
\, \bigr) . 
\]
Thus we can obtain  
\[
\mf{A}_{2, \, 1}^E
\;\,  \overset{\rm bir}{\leadsto} \;\, 
\mf{A}_{1, \, 1}^E . 
\]

\subsubsection{Proof of (3)}

We shall prove assertion (3) of Lemma 4.3. 
Let $A(T)$ be an exponential matrix of $\mf{J}_{[3]}^E$ and write 
\[
 A(T) 
= 
\left(
\begin{array}{c c c}
 1 & \alpha_1(T) & \frac{1}{2} \, \alpha_1(T)^2 + \alpha_2(T) \\
 0 & 1 & \alpha_1(T) \\
 0 & 0 & 1
\end{array} 
\right) 
\qquad 
\bigl( \, 
\alpha_1(T) , \alpha_2(T) \in \sfP , \; \alpha_1(T) \ne 0 
\, \bigr) . 
\]
Consider the birational map $\sigma : \mathbb{P}^2_k \dasharrow \mathbb{P}^2_k$ defined by 
\begin{align*}
\sigma(x_0 \, : \, x_1 \, : \, x_2)
 : = 
\left(\,
x_0 \, x_2 - \frac{1}{2} \, x_1^2 
 \; : \; 
x_1 \, x_2 
 \; : \; 
x_2^2 
\, \right) . 
\end{align*}
This $\sigma$ gives a birational equivalence between $A(T)$ and $B(T)$, 
where 
\[
 B(T) 
:= 
\left(
\begin{array}{c c c }
 1 & 0 & \alpha_2(T) \\
 0 & 1 & \alpha_1(T) \\
 0 & 0 & 1 
\end{array}
\right) . 
\]
Thus we have 
\[
 \mf{J}_{[3]}^E \;\, \overset{\rm bir}{\leadsto} \;\, \mf{A}_{2, \, 1}^E . 
\]

\subsection{Birational equivalence of exponential matrices of 
$(\, \mf{A}_{1, \, 1}^E \,)^\indep$ and $(\, \mf{A}_{1, \, 2}^E \,)^\indep$}

\subsubsection{Birational equivalence of exponential matrices of 
$(\, \mf{A}_{1, \, 1}^E \,)^\indep$}

\begin{lem}
\label{A} 
Let $A(T)$ and $B(T)$ be exponential matrices of $(\, \mf{A}_{1, \, 1}^E \,)^\indep$. 
Write 
\begin{align*}
\begin{array}{r @{\,} l @{\qquad} l}
A(T) 
 & = 
\left(
\begin{array}{c c c}
 1 & 0 & \alpha(T) \\
 0 & 1 & 0  \\
 0 & 0 & 1 
\end{array}
\right) 
& \bigl( \, \alpha(T) \in \sfP \backslash \{\, 0 \, \} \, \bigr) , \\ [2.5em]
B(T) 
 & = 
\left(
\begin{array}{c c c}
 1 & 0 & \beta(T) \\
 0 & 1 & 0  \\
 0 & 0 & 1 
\end{array}
\right) 
& \bigl( \, \beta(T) \in \sfP \backslash \{\, 0 \, \} \, \bigr) . 
\end{array} 
\end{align*}
Then the following conditions {\rm (1)} and {\rm (2)} are equivalent: 
\begin{enumerate}[label = {\rm (\arabic*)}]
\item The exponential matrices $A(T)$ and $B(T)$ are birationally equivalent, i.e., 
$A(T) \overset{\rm bir}{\sim} B(T)$. 

\item There exists an element $\lambda$ of $k \backslash \{ \, 0 \, \}$ 
such that $\alpha(T) = \lambda \, \beta(T)$. 
\end{enumerate} 
\end{lem}

\begin{proof}

We first prove the implication (1) $\Longrightarrow$ (2). 
There exists a birational map $\sigma : \mathbb{P}^2_k \dasharrow \mathbb{P}^2_k$ 
such that the following diagram is commutative: 
\[
\xymatrix@R=36pt@C=36pt@M=6pt{
 \G_a \times \mathbb{P}^2_k \ar[r]^(.6){A(T)} \ar@{-->}[d]_{1 \times \sigma} 
 & \mathbb{P}^2_k \ar@{-->}[d]^\sigma \\
 \G_a \times \mathbb{P}^2_k \ar[r]_(.6){B(T)}
 & \mathbb{P}^2_k 
}
\]
Write 
\[
 \sigma(x_0 \, : \, x_1 \, : \, x_2 ) 
 = 
(f_0 \, : \, f_1 \, : \, f_2) , 
\]
where $f_0, f_1, f_2 \in k[x_0, x_1, x_2]$ are homogenous polynomials of the same degree. 
Thus 
\begin{align*}
& 
\bigl( \, 
f_0(x_0 + \alpha \, x_2 , \, x_1, \, x_2) 
\, : \, 
f_1(x_0 + \alpha \, x_2 , \, x_1, \, x_2) 
\, : \, 
f_2(x_0 + \alpha \, x_2 , \, x_1, \, x_2) 
\, \bigr) \\
& \qquad 
 = 
\bigl( \, 
f_0(x_0, x_1, x_2) + \beta \, f_2(x_0, x_1, x_2) 
\, : \, 
 f_1(x_0, x_1, x_2) 
\, : \, 
 f_2(x_0, x_1, x_2) 
\, \bigr)  . 
\end{align*} 
For all $(x_0 \, : \, x_1 \, : \, x_2) \in \mathbb{P}^2_k$, 
there exists an element $\mu$ of $k \backslash \{ \, 0 \,\}$ such that 
\begin{align*}
& 
\bigl( \, 
f_0(x_0 + \alpha \, x_2 , \, x_1, \, x_2) , \; \;
f_1(x_0 + \alpha \, x_2 , \, x_1, \, x_2) , \; \;
f_2(x_0 + \alpha \, x_2 , \, x_1, \, x_2) 
\, \bigr) \\
& \qquad 
 = 
\mu \,
\bigl( \, 
f_0(x_0, x_1, x_2) + \beta \, f_2(x_0, x_1, x_2) , \; \;
 f_1(x_0, x_1, x_2) , \; \; 
 f_2(x_0, x_1, x_2) 
\, \bigr)  . 
\end{align*} 
Letting $T = 0$, we have $\mu = 1$. 
Thus 
\begin{align}
\tag{\ref{A}.1}
\left\{
\begin{array}{r @{\,} l}
f_0(x_0 + \alpha \, x_2 , \, x_1, \, x_2) 
 & = f_0(x_0, x_1, x_2) + \beta \, f_2(x_0, x_1, x_2) , \\
f_1(x_0 + \alpha \, x_2 , \, x_1, \, x_2) 
 & = f_1(x_0, x_1, x_2) , \\
f_2(x_0 + \alpha \, x_2 , \, x_1, \, x_2) 
 & = f_2(x_0, x_1, x_2) . 
\end{array}
\right. 
\end{align}
From the second and third equalities of $(\ref{A}.1)$, 
we can obtain $f_1(x_0, x_1, x_2) , f_2(x_0, x_1, x_2) \in k[x_1, x_2]$. 
Write 
\begin{align*}
 f_0(x_0, x_1, x_2) = \sum_{i = 0}^n \xi_i(x_1, x_2) \, x_0^i , 
\qquad 
\xi_i (x_1, x_2) \in k[x_1, x_2] \quad ( \, 0 \leq i \leq n \, ) , 
\qquad 
\xi_n(x_1, x_2) \ne 0 .
\end{align*}
So, we have 
\begin{align*}
& f_0(x_0 + \alpha \, x_2 , \, x_1, \, x_2) 
 = f_0(x_0, x_1, x_2) 
 + \sum_{j = 0}^{n - 1} \left(\, \sum_{i = j + 1}^n \binom{i}{j} \, \alpha^{i - j} \cdot \xi_i(x_1, x_2) \, x_2^{i - j} \, \right)  x_0^j . 
\end{align*}
Now the first equality of $(\ref{A}.1)$ implies 
\[
 \sum_{j = 0}^{n - 1} \left(\, \sum_{i = j + 1}^n \binom{i}{j} \,  \alpha^{i - j} \cdot \xi_i(x_1, x_2) \, x_2^{i - j} \, \right)  x_0^j 
 = 
\beta \, f_2 . 
\]
Since $f_2 \in k[x_1, x_2]$, we have 
\begin{align}
\tag{\ref{A}.2}
\left\{ 
\begin{array}{l}
\ds  \sum_{i = 1}^n \alpha^i \cdot \xi_i(x_1, x_2) \, x_2^i = \beta \, f_2 , \\
\ds \sum_{i = j + 1}^n \binom{i}{j} \, \alpha^{i - j} \cdot \xi_i(x_1, x_2) \, x_2^{i - j} = 0 
\qquad \text{ for all \; $1 \leq  j \leq n - 1 $}  . 
\end{array}
\right. 
\end{align}
Consider the highest term in $T$ appearing in both sides of the first equality of (\ref{A}.1). 
So, write 
\begin{align*}
\left\{ 
\begin{array}{r @{\,} l @{\qquad} l @{\qquad} l }
 \alpha & = \ds \sum_{\imath = 0}^e c_\imath \, T^{p^\imath} , 
 & c_\imath \in k \quad (0 \leq \imath \leq e) , 
 & c_e \ne 0 , \\ [1.5em] 
 \beta & = \ds \sum_{\jmath = 0}^f d_\jmath \, T^{p^\jmath} , 
 & d_\jmath \in k \quad (\, 0 \leq \jmath \leq f \,) , 
 & d_f \ne 0 . 
\end{array} 
\right. 
\end{align*}
We have 
\[
 (c_e \, T^{p^e})^n \cdot  \xi_n(x_1, x_2) \, x_2^n 
 = d_f \, T^{p^f} \, f_2 , 
\]
and thereby have $f \geq e$. 
Similarly, consider the birational map $\sigma^{-1}$, which gives an birational equivalence 
between the exponential matrices $B(T)$ and $A(T)$. We also have $e \geq f$. 
Thus
\[
 \deg_T \alpha = \deg_T \beta . 
\]
Now the first equality of $(\ref{A}.2)$ implies $n = 1$ and 
\begin{align}
\tag{\ref{A}.3}
\alpha \cdot  \xi_1(x_1, x_2) \,  x_2 = \beta \, f_2 . 
\end{align}
Thus $\alpha = \lambda \, \beta$ for some $\lambda \in k$ (since $\beta(T) \ne 0$ and 
$\sigma$ is birational).

We next prove the implication (2) $\Longrightarrow$ (1). 
Let $\sigma : \mathbb{P}^2_k \dasharrow \mathbb{P}^2_k$ be the birational map 
defined by 
\[
 \sigma(x_0 \, : \, x_1 \, : \, x_2) 
 = 
(x_0 \, : \, x_1 \, : \, \lambda \, x_2) . 
\]
This $\sigma$ gives an birational equivalence between the exponential matrices 
$A(T)$ and $B(T)$.

\end{proof}

\subsubsection{Birational equivalence of exponential matrices of 
$(\, \mf{A}_{1, \, 2}^E \,)^\indep$}

\begin{lem}
\label{B}
Let $A(T)$ and $B(T)$ be exponential matrices of $(\, \mf{A}_{1, \, 2}^E \,)^\indep$. 
Write 
\begin{align*}
\begin{array}{r @{\,} l @{\qquad} l}
 A(T) 
& = 
\left(
\begin{array}{c c c}
 1 & \alpha_1(T) & \alpha_2(T) \\
 0 & 1 & 0 \\
 0 & 0 & 1
\end{array} 
\right) 
&
\left(
\begin{array}{l}
\alpha_1(T) , \alpha_2(T) \in \sfP , \\
\text{$\alpha_1(T)$ and $\alpha_2(T)$ are linearly independent over $k$}
\end{array}
\right) , \\ [2.5em] 
 B(T) 
& = 
\left(
\begin{array}{c c c}
 1 & \beta_1(T) & \beta_2(T) \\
 0 & 1 & 0 \\
 0 & 0 & 1
\end{array} 
\right) 
&
\left(
\begin{array}{l}
\beta_1(T) , \beta_2(T) \in \sfP , \\
\text{$\beta_1(T)$ and $\beta_2(T)$ are linearly independent over $k$}
\end{array}
\right) . 
\end{array} 
\end{align*}
Then the following conditions {\rm (1)} and {\rm (2)} are equivalent: 
\begin{enumerate}[label = {\rm (\arabic*)}]
\item The exponential matrices $A(T)$ and $B(T)$ are birationally equivalent, i.e., 
$A(T) \overset{\rm bir}{\sim} B(T)$. 

\item There exists a regular matrix $Q$ of $\GL(2, k)$ such that 
\[
\bigl(\, 
 \alpha_1(T) \; \; \alpha_2(T) 
\, \bigr)
 = 
\bigl(\, 
 \beta_1(T) \; \; \beta_2(T) 
\, \bigr)
\; 
Q . 
\]
\end{enumerate} 
\end{lem}

\begin{proof}

We first prove the implication (1) $\Longrightarrow$ (2). 
We consider the case 
\[
\deg_T \alpha_1(T) < \deg_T \alpha_2(T) , \qquad 
\deg_T \beta_1(T) < \deg_T \beta_2(T) . 
\] 
The other cases can be reduced to the above case. 
There exists a birational map $\sigma : \mathbb{P}^2_k \dasharrow \mathbb{P}^2_k$ 
such that the following diagram is commutative: 
\[
\xymatrix@R=36pt@C=36pt@M=6pt{
 \G_a \times \mathbb{P}^2_k \ar[r]^(.6){A(T)} \ar@{-->}[d]_{1 \times \sigma} 
 & \mathbb{P}^2_k \ar@{-->}[d]^\sigma \\
 \G_a \times \mathbb{P}^2_k \ar[r]_(.6){B(T)}
 & \mathbb{P}^2_k 
}
\]
Write 
\[
 \sigma(x_0 \, : \, x_1 \, : \, x_2 ) 
 = 
(f_0 \, : \, f_1 \, : \, f_2) , 
\]
where $f_0, f_1, f_2 \in k[x_0, x_1, x_2]$ are homogenous polynomials of the same degree. 
Thus 
\begin{align*}
& 
\bigl( \, 
f_0(x_0 + \alpha_1 \, x_1 + \alpha_2 \, x_2 , \, x_1, \, x_2) 
\, : \, 
f_1(x_0 + \alpha_1 \, x_1 + \alpha_2 \, x_2 , \, x_1, \, x_2) 
\, : \, 
f_2(x_0 + \alpha_1 \, x_1 + \alpha_2 \, x_2 , \, x_1, \, x_2) 
\, \bigr) \\
& \qquad 
 = 
\bigl( \, 
f_0(x_0, x_1, x_2) + \beta_1 \, f_1(x_0, x_1, x_2) +  \beta_2 \, f_2(x_0, x_1, x_2) 
\, : \, 
 f_1(x_0, x_1, x_2) 
\, : \, 
 f_2(x_0, x_1, x_2) 
\, \bigr)  . 
\end{align*} 
Thus 
\begin{align}
\tag{\ref{B}.1}
\left\{
\begin{array}{r @{\,} l}
f_0(x_0 + \alpha_1 \, x_1 + \alpha_2 \, x_2 , \, x_1, \, x_2) 
 & = f_0(x_0, x_1, x_2) + \beta_1 \, f_1(x_0, x_1, x_2) +  \beta_2 \, f_2(x_0, x_1, x_2)  , \\
f_1(x_0 + \alpha_1 \, x_1 + \alpha_2 \, x_2 , \, x_1, \, x_2) 
 & = f_1(x_0, x_1, x_2) , \\
f_2(x_0 + \alpha_1 \, x_1 + \alpha_2 \, x_2 , \, x_1, \, x_2) 
 & = f_2(x_0, x_1, x_2) . 
\end{array}
\right. 
\end{align}
From the second and third equalities of $(\ref{B}.1)$, 
we can obtain $f_1(x_0, x_1, x_2) , f_2(x_0, x_1, x_2) \in k[x_1, x_2]$. 
Write 
\begin{align*}
 f_0(x_0, x_1, x_2) = \sum_{i = 0}^n \xi_i(x_1, x_2) \, x_0^i , 
\qquad 
\xi_i (x_1, x_2) \in k[x_1, x_2] \quad ( \, 0 \leq i \leq n \, ) , 
\qquad 
\xi_n(x_1, x_2) \ne 0 .  
\end{align*}
So, we have 
\begin{align*}
& f_0(x_0 + \alpha_1 \, x_1 + \alpha_2 \, x_2 , \, x_1, \, x_2) \\
& \qquad = f_0(x_0, x_1, x_2) 
 + \sum_{j = 0}^{n - 1} \left(\, \sum_{i = j + 1}^n 
\binom{i}{j} \, \xi_i(x_1, x_2) \, (\alpha_1 \, x_1 + \alpha_2 \, x_2)^{i - j} \, \right)  x_0^j . 
\end{align*}
Now the first equality of $(\ref{B}.1)$ implies 
\[ 
\sum_{j = 0}^{n - 1} \left(\, \sum_{i = j + 1}^n 
\binom{i}{j} \, \xi_i(x_1, x_2) \, (\alpha_1 \, x_1 + \alpha_2 \, x_2)^{i - j} \, \right)  x_0^j
 = 
\beta_1 \, f_1 + \beta_2 \, f_2 . 
\]
Since $f_1, f_2 \in k[x_1, x_2]$, we have 
\begin{align}
\tag{\ref{B}.2} 
\left\{ 
\begin{array}{l}
\ds  \sum_{i = 1}^n \xi_i(x_1, x_2) \,  (\alpha_1 \, x_1 + \alpha_2 \, x_2)^i 
 = \beta_1 \, f_1 + \beta_2 \, f_2 , \\
\ds \sum_{i = j + 1}^n \binom{i}{j} \, \xi_i(x_1, x_2) \, 
(\alpha_1 \, x_1 + \alpha_2 \, x_2)^{i - j} = 0 
\qquad \text{ for all \; $1 \leq  j \leq n - 1 $}  . 
\end{array}
\right. 
\end{align}
Consider the highest term in $T$ appearing in both sides of the first equality of $(\ref{B}.2)$. 
So, write 
\begin{align*}
\left\{ 
\begin{array}{r @{\,} l @{\qquad} l @{\qquad} l }
 \alpha_2 & = \ds \sum_{\imath = 0}^e c_\imath \, T^{p^\imath} , 
 & c_\imath \in k \quad (0 \leq \imath \leq e) , 
 & c_e \ne 0 , \\ [1.5em] 
 \beta_2 & = \ds \sum_{\jmath = 0}^f d_\jmath \, T^{p^\jmath} , 
 & d_\jmath \in k \quad (\, 0 \leq \jmath \leq f \,) , 
 & d_f \ne 0 . 
\end{array} 
\right. 
\end{align*}
Since $\deg_T \alpha_2 > \deg_T \alpha_1$ and $\deg_T \beta_2 > \deg_T \beta_1$, 
we have 
\[
 \xi_n(x_1, x_2) \, (c_e \, T^{p^e})^n = d_f \, T^{p^f} \, f_2 , 
\]
and thereby have $f \geq e$. 
Similarly, consider the birational map $\sigma^{-1}$, which gives an birational equivalence 
between the exponential matrices $B(T)$ and $A(T)$. We also have $e \geq f$. 
Thus
\[
 \deg_T \alpha_2 = \deg_T \beta_2 . 
\]
Now the first equality of $(\ref{B}.2)$ implies $n = 1$ and 
\begin{align}
\tag{\ref{B}.3}
 \xi_1(x_1, x_2) \, (\alpha_1 \, x_1 + \alpha_2 \, x_2) = \beta_1 \, f_1 + \beta_2 \, f_2 . 
\end{align}
Comparing the coefficients of $T^{p^e}$ of both sides of the above equality, 
we have 
\[
 \xi_1(x_1, x_2) \cdot c_e \cdot x_2 = d_e \cdot f_2 . 
\]
So, letting $\varphi_2 := (c_e / d_e) \, x_2$, we have $f_2 = \xi_1(x_1, x_2) \, \varphi_2$. 
From equality $(\ref{B}.3)$, we can obtain $f_1 = \xi_1(x_1, x_2) \, \varphi_1$ for some 
$\varphi_1 \in k[x_1, x_2]$. 
Now, equality $(\ref{B}.3)$ implies 
\begin{align}
\tag{\ref{B}.4}
 \alpha_1 \, x_1 + \alpha_2 \, x_2 = \beta_1 \, \varphi_1 + \beta_2 \, \varphi_2 . 
\end{align}
So, $\varphi_1$ is a homogenous polynomial of $k[x_1, x_2]$ of degree $1$. 
We can express $\varphi_1$ and $\varphi_2$ as 
\[
\left\{ 
\begin{array}{r @{\,} l  @{\,} l }
 \varphi_1 & = s \, x_1 + & t \, x_2 , \\
 \varphi_2 & =  & v \, x_2 \qquad (\, s, t, v \in k \,) . 
\end{array}
\right. 
\]
Considering the coefficients of $x_1$ and $x_2$ of equality $(\ref{B}.4)$, we have  
\[
\bigl(\, 
 \alpha_1(T) \; \; \alpha_2(T) 
\, \bigr)
 = 
\bigl(\, 
 \beta_1(T) \; \; \beta_2(T) 
\, \bigr)
\; 
Q , 
\]
where 
\[
 Q :=  \left(
\begin{array}{c c}
 s & t \\
 0 & v 
\end{array}
\right) . 
\]
Since $\alpha_1$ and $\alpha_2$ are linearly independent over $k$, 
the matrix $Q$ is regular.

We next prove the implication (2) $\Longrightarrow$ (1). 
Let $\sigma : \mathbb{P}^2_k \dasharrow \mathbb{P}^2_k$ be the 
birational map defined by 
\[
 \sigma (x_0 \, : \, x_1 \, : \, x_2) 
 = 
(x_0 \, : \, x_1 \, : \, x_2)  
\left(
\begin{array}{c | c c}
1 & 0 & 0 \\
\hline 
 0 & \multicolumn{2}{c}{ \multirow{2}{*}{${^t}Q$} }\\
 0 & 
\end{array} 
\right) . 
\]
This birational map $\sigma$ gives an birational equivalence between 
the exponential matrices $A(T)$ and $B(T)$.

\end{proof}

\subsubsection{$(\, \mf{A}_{1, \, 1}^E  \,)^\indep$ and  $( \, \mf{A}_{1, \, 2}^E \,)^\indep$ 
are not birationally equivalent}

\begin{lem}
\label{4.7}
Let $A(T) \in ( \, \mf{A}_{1, \, 1}^E \,)^\indep$ 
and let $B(T) \in (\, \mf{A}_{1, \, 2}^E  \,)^\indep$. 
Then the exponential matrices $A(T)$ and $B(T)$ are not birationally equivalent. 
\end{lem}

\begin{proof} 
We can experss $A(T)$ and $B(T)$ as 
\begin{align*}
\begin{array}{r @{\,} l @{\qquad} l}  
A(T) 
& = 
\left(
\begin{array}{c c c}
 1 & 0 & \alpha(T) \\
 0 & 1 & 0 \\
 0 & 0 & 1
\end{array} 
\right) 
 & 
\left( 
\begin{array}{l}
\alpha \in \sfP , \quad 
\alpha\ne 0 
\end{array} 
\right) , 
\\ [2.5em] 
 B(T) 
& = 
\left(
\begin{array}{c c c}
 1 & \beta_1(T) & \beta_2(T) \\
 0 & 1 & 0 \\
 0 & 0 & 1
\end{array} 
\right) 
&
\left(
\begin{array}{l}
\beta_1(T) , \beta_2(T) \in \sfP , \\
\text{$\beta_1(T)$ and $\beta_2(T)$ are linearly independent over $k$}
\end{array}
\right) . 
\end{array}
\end{align*}
Suppose to the contrary that $A(T)$ and $B(T)$ are birationally equivalent. 
So, there exists a birational map $\sigma : \mathbb{P}^2_k \dasharrow \mathbb{P}^2_k$ 
such that the following diagram is commutative: 
\[
\xymatrix@R=36pt@C=36pt@M=6pt{
 \G_a \times \mathbb{P}^2_k \ar[r]^(.6){A(T)} \ar@{-->}[d]_{1 \times \sigma} 
 & \mathbb{P}^2_k \ar@{-->}[d]^\sigma \\
 \G_a \times \mathbb{P}^2_k \ar[r]_(.6){B(T)}
 & \mathbb{P}^2_k 
}
\]
We consider the case 
\[
\deg_T \beta_1(T) < \deg_T \beta_2(T) . 
\] 
The other cases can be reduced to the above case. 
Write 
\[
 \sigma(x_0 \, : \, x_1 \, : \, x_2 ) 
 = 
(f_0 \, : \, f_1 \, : \, f_2) , 
\]
where $f_0, f_1, f_2 \in k[x_0, x_1, x_2]$ are homogenous polynomials of the same degree. 
We can obtain $f_1, f_2 \in k[x_1, x_2]$ and 
\[
 f_0 + \beta_1 \, f_1 + \beta_2 \, f_2 
 =
 f_0(x_0 + \alpha \, x_2, x_1, x_2) , 
\]
which implies 
\begin{align}
\tag{\ref{4.7}.1}
\deg_T \beta_2 \geq \deg_T \alpha . 
\end{align}
Consider $\sigma^{-1}$, which is a birational equivalence between $B(T)$ and $A(T)$. 
Write 
\[
 \sigma^{-1}(x_0 \, : \, x_1 \, : \, x_2 ) 
 = 
(g_0 \, : \, g_1 \, : \, g_2) , 
\]
where $g_0, g_1, g_2 \in k[x_0, x_1, x_2]$ are homogenous polynomials of the same degree.  
We can express $g_0$ as 
\begin{align*}
 g_0(x_0, x_1, x_2) = \sum_{i = 0}^n \xi_i(x_1, x_2) \, x_0^i , 
\qquad 
\xi_i (x_1, x_2) \in k[x_1, x_2] \quad ( \, 0 \leq i \leq n \, ) , 
\qquad 
\xi_n(x_1, x_2) \ne 0 .  
\end{align*}
We can obtain $g_1, g_2 \in k[x_1, x_2]$ and 
\begin{align*}
\ds  \sum_{i = 1}^n \xi_i(x_1, x_2) \,  (\beta_1 \, x_1 + \beta_2 \, x_2)^i 
 = \alpha \, g_2 . 
\end{align*}
So, we have 
\begin{align}
\tag{\ref{4.7}.2}
\deg_T \beta_2 \leq \deg_T \alpha . 
\end{align}
By (\ref{4.7}.1) and (\ref{4.7}.2), we obtain $\deg_T \beta_2 = \deg_T \alpha$. 
Therefore,  
\[
 \xi_1(x_1, x_2) \,  (\beta_1 \, x_1 + \beta_2 \, x_2) 
 = \alpha \, g_2 , 
\] 
which implies $g_2 = \xi_1(x_1, x_2) \cdot \lambda \, x_2$ for some $\lambda \in k \backslash \{ \, 0 \, \}$. 
Thus 
\[
 \beta_1 \, x_1 + \beta_2 \, x_2 = \lambda \, \alpha \, x_2 . 
\]
Comparing the coefficients of $x_1$ in both sides of the above equality, 
we have $\beta_1 = 0$. 
This is a contradiction. 

\end{proof}

\subsection{Proof of Theorem 0.3}

We shall prove Theorem 0.3 by separating the cases $p = 2$ and $p \geq 3$. 
\medskip 

\noindent 
{\bf The case $p = 2$.} 
By Lemmas 4.1 and 4.3, we have 
\begin{align*}
 \Mat(3, k[T])^E 
 & \leadsto 
 \mf{A}_{1, \, 2}^E \;  \cup \; \mf{A}_{2, \, 1}^E \\
 & \overset{\rm bir}{\leadsto}  
 \bigl( \, ( \, \mf{A}_{1, \, 2}^E \,)^\indep \, \sqcup \, (\, \mf{A}_{1, \, 1}^E  \,)^\indep 
\, \sqcup \, \{ \, I_3 \,\} \, \bigr) 
 \; \cup \; ( \, \mf{A}_{1, \, 1}^E \, )  \\
 & = 
  ( \, \mf{A}_{1, \, 2}^E \,)^\indep \, \sqcup \, (\, \mf{A}_{1, \, 1}^E  \,)^\indep 
\, \sqcup \, \{ \, I_3 \,\} . 
\end{align*}
Consider any two of the above disjoint three sets. 
Choose exponential matrices one-by-one from the two sets. 
By Lemmas 1.1 and 4.7,  the chosen two exponential matrices are 
not birationally equivalent. 
Thus 
\begin{align*}
\Mat(3, k[T])^E / \overset{\rm bir}{\sim} 
 & = 
\left( \; 
  ( \, \mf{A}_{1, \, 2}^E \,)^\indep \, \sqcup \, (\, \mf{A}_{1, \, 1}^E  \,)^\indep 
\, \sqcup \, \{ \, I_3 \,\} 
\; \right) 
/ \overset{\rm bir}{\sim} \\
 & = 
  ( \, \mf{A}_{1, \, 2}^E \,)^\indep / \overset{\rm bir}{\sim}
 \; \; \sqcup \; \; (\, \mf{A}_{1, \, 1}^E  \,)^\indep  / \overset{\rm bir}{\sim}
 \; \; \sqcup \; \; \{ \, I_3 \,\} . 
\end{align*}
By Lemmas \ref{A} and \ref{B}, we can obtain a one-to-one correspondence 
between the above set 
and the set 
\[
 \{\, 0 \,\} 
\; \; \sqcup \; \; 
\bigl( \; 
 (\, \sfP \,)^\indep / \sim_{\GL(1, k)} 
\; \bigr)
\; \; \sqcup \; \; 
\bigl(\; 
(\, \sfP^2 \,)^\indep / \sim_{\GL(2, k)}
\; \bigr) . 
\]
\medskip

\noindent 
{\bf The case $p \geq 3$.} 
By Lemmas 4.1 and 4.3, we have 
\begin{align*}
 \Mat(3, k[T])^E 
 & \leadsto 
 \mf{A}_{1, \, 2}^E \; \cup \; \mf{A}_{2, \, 1}^E \; \cup \; \mf{J}_{[3]}^E \\
 & \overset{\rm bir}{\leadsto}  
 \bigl( \, ( \, \mf{A}_{1, \, 2}^E \,)^\indep \, \sqcup \, (\, \mf{A}_{1, \, 1}^E  \,)^\indep 
\, \sqcup \, \{ \, I_3 \,\} \, \bigr) \\
& \qquad 
 \; \cup \; ( \, \mf{A}_{1, \, 1}^E \, ) 
 \; \cup \;   ( \, \mf{A}_{2, \, 1}^E \, ) \\
 & \overset{\rm bir}{\leadsto}  
  ( \, \mf{A}_{1, \, 2}^E \,)^\indep \, \sqcup \, (\, \mf{A}_{1, \, 1}^E  \,)^\indep 
\, \sqcup \, \{ \, I_3 \,\} . 
\end{align*}
Thus we can obtain the desired one-to-one correspondence.

\end{document}